\documentclass[12pt,oneside,a4paper]{amsart}
\pdfoutput=1
\usepackage{amssymb}

\usepackage[T1]{fontenc}
\usepackage{newtxtext,newtxmath}

\renewcommand{\epsilon}{\varepsilon}
\renewcommand{\phi}{\varphi}

\newtheorem{lemma}{Lemma}
\newtheorem{theorem}{Theorem}
\newtheorem{definition}{Definition}
\newtheorem{remark}{Remark}

\makeatletter
\renewcommand\subsection{\@startsection{subsection}{2}%
  \z@{.5\linespacing\@plus.7\linespacing}{.1\linespacing}%
  {\normalfont\scshape}}
\makeatother

\title[$C^{1,\alpha}$ and $C^{2,\alpha}$ estimates for GJEs]{First and Second derivative H\"older estimates for generated Jacobian equations}
\author{Cale Rankin}
\email{}
\thanks{This research is supported by ARC DP 200101084}

\begin{document}
\maketitle
\begin{abstract}
We prove two H\"older regularity results for solutions of generated Jacobian equations. First, that under the A3 condition and the assumption of nonnegative $L^p$ valued data solutions are $C^{1,\alpha}$ for an $\alpha$ that is sharp. Then, under the additional assumption of positive Dini continuous data, we prove a $C^{2}$ estimate. Thus the equation is uniformly elliptic and when the data is H\"older continuous solutions are in $C^{2,\alpha}$. 
\end{abstract}

\setcounter{tocdepth}{1}
\tableofcontents

\section{Introduction}
\label{sec:introduction}

Generated Jacobian equations are a class of PDE which model problems in geometric optics and have recently seen applications to  monopolist problems in economics \cite{MR3121636,MR3957395}. These equations are of the form
\begin{equation}
  \label{eq:gje}
   \det DY(\cdot,u,Du) = \psi(\cdot,u,Du) \text{ in }\Omega,
\end{equation}
where $Y:\mathbf{R}^n \times \mathbf{R} \times \mathbf{R}^n \rightarrow \mathbf{R}^n$, $\Omega$ is a domain, and $u:\Omega \rightarrow \mathbf{R}$ satisfies a convexity condition that ensures the PDE is elliptic.

The precise form of this condition on $u$ and the structure of $Y$ requires a number of definitions which we introduce in Section \ref{sec:gen-func}. However, the framework we work in includes, in addition to the just listed applications, two well-known special cases. First, the Monge--Amp\`ere equation (take $Y(\cdot,u,Du) = Du$).  Second, the Monge--Amp\`ere type equations from optimal transport (take $Y(\cdot,u,Du)=Y(\cdot,Du)$ as the optimal transport map depending on the  corresponding potential function $u$ for the optimal transport problem). Indeed, generated Jacobian equations (GJEs) were introduced to treat the aforementioned new applications using the techniques from optimal transport. That's what we do here: Two fundamental results for the regularity theory in optimal transport are the $C^{1,\alpha}$ result of Liu \cite{MR2476419} and the $C^{2,\alpha}$ results of Liu, Trudinger, Wang \cite{MR2748621}. These are based on the corresponding results of Caffarelli in the Monge--Amp\`ere case, respectively, \cite{MR1038359} and \cite{MR1038360}. In this paper we extend the results of \cite{MR2476419,MR2748621} to generated Jacobian equations. Our main results are stated precisely at the conclusion of Section \ref{sec:gen-func}, once the necessary definitions have been introduced. The importance of our results is that, first, the $C^{1,\alpha}$ (or at least some $C^1$) result is a necessary assumptions in the derivation of models in geometric optics and, second, the $C^{2,\alpha}$ result puts us in the regime of classical elliptic PDE and lets us bootstrap higher regularity. 

We note that the corresponding $C^{1,\alpha}$ result of Loeper in the optimal transport setting \cite{MR2506751} has been extended to generated Jacobian equations by Jeong \cite{MR4265360}. Thus our key contribution for the $C^{1,\alpha}$ result is improving the value of $\alpha$ so that it is sharp. As we explain at the conclusion of Section \ref{sec:gen-func} this yields a corresponding improvement on the $C^{2,\alpha}$ result. An outline of the paper is provided by the table of contents. 

\section{Generating functions, $g$-convexity, and GJEs}
\label{sec:gen-func}

In this section we state the essential definitions. Further introductory material can be found in the expository article of Guillen \cite{MR3967932}, and more detailed outlines of the whole theory in \cite{MR3121636,MR4218157,MR3639322,RankinThesis}.

The framework for GJEs was introduced by Trudinger \cite{MR3121636} and is built around a generalized notion of convexity.  The generating function, which we now define, plays a central role, essentially that of affine hyperplanes in classical convexity. 

\begin{definition}
  A generating function is a function, which we denote by $g$, satisfying the conditions A0,A1,A1$^*$, and A2. 
\end{definition}

\textbf{A0.} $g \in C^4(\overline{\Gamma})$ where $\Gamma$ is a bounded domain of the form
\[ \Gamma = \{(x,y,z) ; x \in U, y \in V,z \in I_{x,y}\},\]
for domains $U,V \subset \mathbf{R}^n$ and $I_{x,y}$ an open interval for each $(x,y) \in U \times V$. Moreover we assume there is an open interval $J$ such that $g(x,y,I_{x,y}) \supset J$ for each $(x,y) \in U \times V$. 

\textbf{A1.} For each $(x,u,p) \in \mathcal{U}$ defined by
\[\mathcal{U} = \{(x,g(x,y,z),g_x(x,y,z)) ; (x,y,z) \in \Gamma\}, \]
there is a unique $(x,y,z) \in \Gamma$, whose $y,z$ components we denote by $Y(x,u,p)$, $Z(x,u,p)$, satisfying
\begin{align}
 \label{eq:yzdef} &g(x,y,z) = u &&g_x(x,y,z) = p.
\end{align}

\textbf{A1$^*$.} For each fixed $y,z$ the mapping $x \mapsto \frac{g_y}{g_z}(x,y,z)$ is injective on its domain of definition. 

\textbf{A2.} On $\overline{\Gamma}$ there holds $g_z<0$ and the matrix
\[ E:=g_{i,j} - g_{z}^{-1}g_{i,z}g_{,j}\]
satisfies $\det E \neq 0.$ Here subscripts before and after a comma denote, respectively, differentiation in $x$ and $y$.

Two examples are $g(x,y,z) = x\cdot y - z$ which generates (in accordance with Definition \ref{defn:gje}) the Monge--Amp\`ere equation and standard convexity, and $g(x,y,z) = c(x,y) - z$, where $c$ is a cost function from optimal transport, which generates the Monge--Amp\`ere type equation from optimal transport, and the cost convexity theory \cite[Chapter 5]{MR2459454}.
By a duality structure, which we do not need and thus don't introduce here, the condition A1$^*$ is dual to $A1$, thereby justifying the name. The A0 condition is weakened by some authors who treat $C^2$ generating functions \cite{MR3639322,MR3797806}. However our interior Pogorelov estimate, an essential tool for the $C^{2,\alpha}$ result, is obtained by differentiating the PDE twice which relies on a $C^4$ generating function. 

\begin{definition}
  Let  $\Omega \subset U$ be a domain and $u \in C^0(\overline{\Omega})$. We call $u$ $g$-convex provided for every $x_0 \in \Omega$ there is $y_0,z_0$ such that
  \begin{align}
\label{eq:supp-eq}    &u(x_0) =  g(x_0,y_0,z_0), \\
    \label{eq:strict}    &u(x) \geq g(x,y_0,z_0) \text{ for all $x \in \Omega$ satisfying $x \neq x_0$},\\
\label{eq:contain}    &g(\overline{\Omega},y_0,z_0) \subset J.
  \end{align}
 When the inequality in \eqref{eq:strict} is strict $u$ is called strictly $g$-convex. When the function $x \mapsto g(x,y_0,z_0)$ satisfies \eqref{eq:supp-eq},\eqref{eq:strict}, and \eqref{eq:contain} it is called a $g$-support of $u$ at $x_0$.
\end{definition}
The set of all $y$ such that there is $z$ for which $g(\cdot,y,z)$ is a $g$-support at $x$ is denoted $Yu(x)$. When $u$ is differentiable and $g(\cdot,y,z)$ is a $g$-support at $x$ then $u-g(\cdot,y,z)$ has a minimum at $x$ implying by \eqref{eq:yzdef} that $y = Y(x,u,Du)$.  Similarly, if $g(\cdot,y,z)$ is a $g$-support at $x$ and $u$ is $C^2$ then $D^2u(x) - g_{xx}(x,y,z)$ is a nonnegative definite matrix.  When $u$ is not differentiable at $x$, $Yu(x)$ is not a singleton.

\begin{definition}\label{defn:gje}
  A generated Jacobian equation is an equation of the form \eqref{eq:gje} where the mapping $Y$ derives from a generating function as in the A1 condition. 
\end{definition}
For a GJE to make sense we must have $(\text{Id},u,Du)(\Omega) \subset \mathcal{U}$. By calculations which are now standard \cite{MR3121636}, a $C^2$ solution of \eqref{eq:gje} satisfies the Monge--Amp\`ere type equation
\begin{equation}
  \label{eq:mate}
  \det[D^2u - A(\cdot,u,Du)] = B(\cdot,u,Du) \text{ in }\Omega,
\end{equation}
where $A,B$ are defined on $U$ by 
\begin{align}
  \label{eq:adef} A_{ij}(x,u,p) &= g_{ij}(x,Y(x,u,p),Z(x,u,p)),\\
  \label{eq:bdef} B(x,u,p) &= \det E(x,u,p) \psi(x,u,p).
\end{align}
This equation is degenerate elliptic provided $u$ is $g$-convex. 

The following definition extends the definition of Aleksandrov solution for the Monge--Amp\`ere equation to generated Jacobian equations.
\begin{definition}
  A $g$-convex function $u:\Omega \rightarrow\mathbf{R}$ is called an Aleksandrov solution of
  \begin{equation}
    \label{eq:gje-def}
    \det DY(\cdot,u,Du) = \nu \text{ in }\Omega,
  \end{equation}
  for $\nu$ a Borel measure on $\Omega$, provided for every Borel $E \subset \Omega$
  \[ |Yu(E)| = \nu(E).\]
\end{definition}

Whilst \eqref{eq:gje-def} would, classically, require a $C^2$ function, we've defined Aleksandrov solutions for merely $g$-convex functions. However it is a consequence of the change of variables formula and the Lebesgue differentiation theorem that $C^2$ Aleksandrov solutions are classical solutions.

We introduce one final condition on the generating function. It was introduced by Ma, Trudinger, and Wang \cite{MR2188047} and was  extended to GJEs in \cite{MR3121636}. The necessity of the weakened form, for even $C^1$ regularity, was proved by Loeper \cite{MR2506751}. 

\textbf{A3.} There is $c>0$ such that 
\[ D_{p_kp_l}A_{ij}(x,u,p)\xi_i\xi_j\eta_k\eta_l \geq c,\]
for all unit vectors $\xi,\eta$ satisfying $\xi \cdot \eta = 0$. \\
The A3 weak (A3w) condition, is the same but with $c=0$. 

\subsection{Statement of main theorems}
\label{sec:stat-main-theor}

\begin{theorem}\label{thm:main1}
  Let $g$ be a generating function satisfying A3. Let $u:\Omega \rightarrow \mathbf{R}$ be a $g$-convex Aleksandrov solution of \eqref{eq:gje-def} with $\nu = f \ dx \in L^p(U)$ for $p > (n+1)/2$. Then $u \in C^{1,\alpha}(\Omega)$ with
  \[ \alpha = \frac{\beta(n+1)}{2n^2+\beta(n-1)} \text{ where } \beta = 1-\frac{n+1}{2p}.\]
\end{theorem}
\begin{remark}
  Liu proved this value, $\alpha = (2n-1)^{-1}$ when $p=\infty$, is sharp. That is, there exists a function $u$ which solves a GJE satisfying the hypothesis of the theorem, and which is in $C^{1,\alpha}$ for the stated $\alpha$, but not in $C^{1,\alpha+\epsilon}$ for any $\epsilon > 0$.
  We note that Loeper \cite{MR2506751} and Jeong \cite{MR4265360} have proved the H\"older regularity for the right-hand side a measure $\nu$ satisfying $\nu(B_\epsilon(x)) \leq C \epsilon^{n(1-1/p)}$. Our proof is easily adapted to this condition (which is more general than above but comes at the expense of a smaller $\alpha$). We indicate the necessary changes in a remark after the proof of Theorem \ref{thm:main1}. 
\end{remark}

Our $C^2$ estimate is for $\nu = f \ dx $ where $f$ is Dini continuous.
\begin{definition}[Dini continuity] \label{defn:dini}
  Let $f:\Omega \rightarrow\mathbf{R}$. The oscillation of $f$ is
\[ \omega_f(r) := \sup\{|f(x)-f(y)| ; x,y \in \Omega \text{ with }|x-y|< r\}. \]
Then $f$ is called Dini continuous if
\[ \int_0^1 \frac{\omega_f(r)}{r} \ dr < \infty. \]
\end{definition}

\begin{theorem}\label{thm:main2}
Let $g$ be a generating function satisfying A3. Let $u:\Omega \rightarrow \mathbf{R}$ be an Aleksandrov solution of \eqref{eq:gje-def} with $\nu = f \ dx $. If $\lambda \leq f \leq \Lambda$ and $f$ is Dini continuous then $u \in C^2(\Omega)$.  If $f \in C^\alpha(\Omega)$ then $u \in C^{2,\alpha}_{loc}(\Omega)$. 
\end{theorem}
\begin{remark}
  Our result can be stated more precisely as follows. Under the above hypothesis for each $\Omega' \subset \subset \Omega$ there is $C>0$ depending on $g,\Omega',\Omega,\Vert u \Vert_{C^1(\Omega)}$ such that
  \begin{enumerate}
  \item If $f$ is Dini continuous we have
    \begin{equation}
      \label{eq:dini}
      |D^2u(x)-D^2u(y)| \leq C\left[d+\int_0^d\frac{\omega_f(r)}{r} \ dr + d \int_d^1\frac{\omega_f(r)}{r^2} \ dr\right],
      \end{equation}
      where $x,y \in \Omega'$ and $d:= |x-y|$.\\
    \item If $f\in C^\alpha(\Omega)$ for some $\alpha \in (0,1)$ then $u \in C^{2,\alpha}(\Omega)$ with
      \begin{equation}
        \label{eq:holder}
\Vert u \Vert _{C^{2,\alpha}}(\Omega') \leq C+\frac{C}{\alpha(1-\alpha)}\Vert f \Vert_{C^\alpha(\Omega)}.
      \end{equation}
    \end{enumerate}
    We do not prove \eqref{eq:dini} and \eqref{eq:holder} directly --- we just prove a $C^2$ estimate. This ensures the equation is uniformly elliptic then \eqref{eq:dini} and \eqref{eq:holder} follow from \cite[Theorem 3.1]{MR2273802} (details in Section \ref{sec:c2-alpha-regularity}).
  \end{remark}
  \begin{remark}
    A common form of $\nu$ arising in applications is
    \[ \nu = \frac{f(x)}{f^*(Y(x,u,Du))} \ dx.\]
    Theorem \ref{thm:main1} applies whenever $f \in L^p$ and $f^*> \lambda $ for some positive constant $\lambda$. Then the assumption of H\"older continuous right-hand side in Theorem \ref{thm:main2} is satisfied when $\lambda < f,f^* < \Lambda$ and $f,f*$ are H\"older continuous. More precisely if $f \in C^{\beta}, f^* \in C^{\gamma}$ then Theorem \ref{thm:main1} implies the sharp exponent with $p = \infty$ and we subsequently obtain $u \in C^{2,\alpha'}$ for $\alpha' = \text{min}\{\beta,\gamma/(2n-1)\}$.
  \end{remark}

  As is standard we prove apriori estimates for smooth solutions. The results then hold by approximation and uniqueness of the Dirichlet problem in the small \cite[Lemma 4.6]{MR3121636}. The results above all hold without boundary conditions, that is they are interior local results. This is possible because we are considering Aleksandrov solutions under A3. With A3 weakened to A3w such results are not possible. Moreover for applications to optimal transport to conclude that a potential function is an Aleksandrov solution we require also a boundary condition - the second boundary value problem with a target satisfying a convexity condition \cite{MR2188047}.  

\section{Background results and normalization lemma}
\label{sec:background-results}

We will use a number of background results. These originally appeared in \cite{MR3639322} though we'll use the formulation in \cite{Rankin2021}.

We assume $g$ is a generating function satisfying A3w and $u$ is a strictly $g$-convex function with $\det DYu =: \nu$ in the Aleksandrov sense.  Let $x_0$ be given and $g(\cdot,y_0,z_0)$ a support at $x_0$ put $z_h = z_0-h$ where $h>0$ and define the section
\[ S_h = S_h(x_0) := \{ u < g(\cdot,y_0,z_h)\},\]
which by the strict $g$-convexity is compactly contained in $\Omega$ for sufficiently small $h$. 

A lemma \cite[Theorems 3,5]{Rankin2021}, which we employ repeatedly is the following. 
\begin{lemma}\label{lem:sec1}
  There exists $C,d>0$ which depend only on $g$, such that if $\text{diam}(S_h) < d$ and  $\nu$ is a doubling measure, then
  \[C^{-1}|S_h|\nu(S_h) \leq \sup_{S_h}|u(\cdot)-g(\cdot,y_0,z_h)|^n \leq C|S_h|\nu(S_h).  \]
\end{lemma}
Note the requirement that $\nu$ is a doubling measure is only necessary for the lower bound. Also, $
C^{-1}h\leq  \sup_{\Omega}|u(\cdot)-g(\cdot,y_0,z_h)| \leq Ch$ for $C > 0$ depending only on $\sup |g_z|, \inf |g_z|$.  
In the special case where $\nu = f \ dx $ and $\lambda \leq f \leq \Lambda$ we have 
\begin{equation}
  \label{eq:sec-est}
   C^{-1}\lambda|S_h|^2 \leq h^{n} \leq C\Lambda|S_h|^2.  
 \end{equation}

 We introduce new coordinates
\begin{align}
 \label{eq:x-trans} \tilde{x}&:= \frac{g_y}{g_z}(x,y_0,z_h),
\end{align}
When A3w is satisfied $S_h(x_0)$ is convex in the $\tilde{x}$ coordinates \cite[Lemma 2.3]{MR4218157}. 
We often use this result in conjunction with the minimum ellipsoid. The minimum ellipsoid of an arbitrary open convex set $U$ is the unique ellipsoid of minimal volume, denoted $E$, containing $U$. It satisfies $\frac{1}{n}E \subset U\subset E$ where this is dilation with respect to the centre of the ellipsoid \cite[Lemma 2.1]{MR3381500}. We assume, after a rotation and translation that the minimum ellipsoid of $S_h$ is
\begin{equation}
      \label{eq:min-ellipsoid}
 E = \left\{x ; \sum_{i=1}^n \frac{x_i^2}{r_i^2} \leq 1\right\}, \text{ where }r_1 \geq \dots \geq r_n.
\end{equation}
Then elementary convex geometry implies $c_n r_1\dots r_n \leq |S_h| \leq C_n r_1 \dots r_n$ and \eqref{eq:sec-est} becomes
\begin{equation}
  \label{eq:sec-ellips}
  C^{-1} r_1 \dots r_n \leq h^{n/2} \leq Cr_1 \dots r_n. 
\end{equation}

We also define here the good shape constant. Let $U$ be any convex set, and assume its minimum ellipsoid is given by \eqref{eq:min-ellipsoid}. Then \textit{a} good shape constant is any $C$ satisfying $C \geq r_1/r_n$ and \textit{the} good shape constant is just $r_1/r_n$, explicitly, the infimum of all good shape constants. For solutions of generated Jacobian equations the good shape constant of $S_h(x_0)$ carries information about $|D^2u(x_0)|$ (see Lemmas \ref{lem:deriv-imp-shape} and \ref{lem:shape-imp-deriv}). 

When A3w is strengthened to A3 we have a particularly strong estimate concerning the geometry of sections and their height. It is used repeatedly throughout this paper. In optimal transport this estimate is due to Liu \cite[Lemma 4]{MR2476419} and we largely follow his proof. 

\begin{lemma}\label{lem:lemma-shape-sections}
  Assume $g$ is a generating function satisfying A3 and $u:\Omega \rightarrow \mathbf{R}$ is a $C^2$ $g$-convex function. Assume that $S_h(x_0) \subset \subset \Omega$ and that the minimum ellipsoid of $S_h(x_0)$ (in the $\tilde{x}$ coordinates) is \eqref{eq:min-ellipsoid}.  Then there is $C$ depending only on $g$ such that
  \[ \frac{hr_1^2}{r_n^2} \leq C. \]
\end{lemma}
\begin{proof}
  We work in the $\tilde{x}$ coordinates, though keep the notation $x$. Define $T:\mathbf{R}^n \rightarrow \mathbf{R}^n$ by
  \begin{equation}
    \label{eq:t-def}
    T(x) = (x_1/r_1,\dots,x_n/r_n).
  \end{equation}
    Note
  \begin{equation}
    \label{eq:norm}
     B_{1/n} \subset U:=TS_h \subset B_1.
  \end{equation}
  Let $\hat{x}$ be the boundary point of $U$ on the negative $x_n$ axis. In a neighbourhood of $\hat{x}$, denoted $\mathcal{N}$, we represent $\partial \Omega$ as a graph of some function $\rho$, that is
  \[ \partial U \cap \mathcal{N}= \{ (x',\rho(x')) ; x' = (x_1,\dots,x_{n-1})\}.\]
  Using \eqref{eq:norm} we may assume $\rho$ is defined for $|x'| < 1/n$. Similarly by \eqref{eq:norm} and the convexity of $U$ we conclude $|D\rho| \leq c(n)$ when $|x'| \leq 1/2n$. Let $\gamma$ be the curve
  \[ \gamma = \{ \gamma(t) = (t,0,\dots,0,\rho(te_1)); -1/4n < t < 1/4n\}. \]
  Because \eqref{eq:norm} implies $|\rho_{11}(\gamma(t))| \leq C(n)$ at some $t \in (-1/4n,1/4n)$  the proof will be complete provided we can show that
  \begin{equation}
    \label{eq:proof1-toshow}
     C|\rho_{11}(\gamma(t))| \geq hr_1^2/r_n^2,
   \end{equation}
   for every $t \in (-1/4n,1/4n)$.

  Let $v:U \rightarrow \mathbf{R}$ be defined on $\overline{x} \in U$ by
  \[ v(\overline{x}) = u(T^{-1}\overline{x}) - g(T^{-1}\overline{x},y,z_h).\]
 Differentiating $v(\gamma(t)) = 0$ once, then twice, with respect to $t$ gives
  \begin{align}
    \label{eq:orthog-reln} D_{k}v \dot{\gamma}_k &= 0,\\
    \label{eq:curv-reln} D_{\dot{\gamma}\dot{\gamma}}v &= -D_{k}v\ddot{\gamma}_k =  -D_nv\rho_{11}. 
  \end{align}
  To estimate $D_{\dot{\gamma}\dot{\gamma}}v =  D_{kl}v\dot{\gamma}_k\dot{\gamma}_l $ we compute at $\overline{x} = Tx$
  \begin{align*}
    D_{kl}v &= [ u_{kl}-g_{kl}(x,y_0,z_h)]r_kr_l\\
                                     &= [ u_{kl}-g_{kl}(x,Yu(x),Zu(x))]r_kr_l \\
            &\quad+ [g_{kl}(x,Yu(x),Zu(x))-g_{kl}(x,y_0,z_h)]r_kr_l\\
    &\geq [g_{kl}(x,Yu(x),Zu(x))-g_{kl}(x,y_0,z_h)]r_kr_l.
  \end{align*}
  The inequality is by $g$-convexity of $u$. 
 Put $g_h(\cdot):= g(\cdot,y_0,z_h)$. Then use the definition of $A$ (equation \eqref{eq:adef}), $u(x) = g_h(x)$ on $\partial S_h$, and a Taylor series to obtain
  \begin{align}
\nonumber    D_{kl}v &\geq [A_{kl}(x,u(x),Du(x)) - A_{kl}(x,u(x),Dg_h(x))  ] r_kr_l\\
    \label{eq:10}                                  &= A_{kl,p_m}(x,u(x),Dg_h(x))D_m(u-g_h)r_kr_l\\
                                     &\quad\quad+A_{kl,p_mp_n}(x,u(x),p_\tau)D_m(u-g_h)D_n(u-g_h)r_kr_l,
  \end{align}
  where $p_\tau = \tau Du(x)+ (1-\tau)Dg_h(x)$ for some $\tau \in (0,1)$.  A direct, but involved, calculation which we relegate to Appendix \ref{sec:vanishing} implies
  \[  A_{kl,p_m}(x,u(x),Dg_h(x))D_m(u-g_h)r_kr_l \dot{\gamma}_k\dot{\gamma}_l = 0.\]
  Thus, using also $D_m(u-g_h) = D_mv/r_m$, \eqref{eq:10} becomes
  \[    D_{kl}v  \geq A_{kl,p_mp_n}D_mvD_nv\frac{r_kr_l}{r_mr_n}.  \]

  Now returning to $D_{\dot{\gamma}\dot{\gamma}}v$ we have
  \begin{equation}
   D_{\dot{\gamma}\dot{\gamma}}v  \geq A_{kl,p_mp_n}\frac{D_mv}{r_m}\frac{D_nv}{r_n}(\dot{\gamma}_kr_k)(\dot{\gamma}_lr_l).\label{eq:dgg}
    \end{equation}
    Since by \eqref{eq:orthog-reln} $\dot{\gamma}$ is orthogonal to $Dv$ we also have orthogonality of $\xi:= (r_1\dot{\gamma}_1,\dots,r_n\dot{\gamma}_n)$ and $\eta := (D_1v/r_1,\dots,D_nv/r_n)$. Thus employing the A3 condition in \eqref{eq:dgg} yields
    \[  D_{\dot{\gamma}\dot{\gamma}}v \geq c|\xi|^2|\eta|^2 \geq c\frac{|D_nv|^2r_1^2}{r_n^2}.\]
    Now we substitute this into \eqref{eq:curv-reln} to obtain
    \[ \rho_{11} \geq c\frac{|D_nv|r_1^2}{r_n^2}.\]
    If we can show  $|D_nv| \geq Ch$ then we've obtained \eqref{eq:proof1-toshow}.

 For this final estimate fix $x_1 \in \partial S_h$ and set
  \[ h(\theta) := g(x_\theta,y_1,z_1) - g(x_\theta,y_0,z_h) \text{ where }x_\theta = \theta x_1 + (1-\theta)x_0,\]
  where $g(\cdot,y_1,z_1)$ supports $u$ at $x_1$ so $Du(x_1) = g_x(x_1,y_1,z_1)$.  A standard argument \cite[Eq. A.14]{RankinThesis} using the A3w condition implies that for $K$ depending only on $g$, 
  \[ h''(\theta) \geq - K|h'(\theta)|.\]
  Then we follow \cite[Eq. 19]{MR4308251} (there are similar arguments in \cite{MR2459454,MR3639322,MR4218157,MR3121636}) to obtain
  \[ h'(t_1) \leq e^{K(t_2-t_1)}h'(t_2),\]
  for $t_1 < t_2$. Choosing $t_2 = 1$ and integrating from $t_1 = 0$ to $1$ we have
  \begin{equation}
    \label{eq:h}
       -h(0) \leq C(K) h'(1).
  \end{equation}
  Now $-h(0) = g(x_0,y_0,z_h) - g(x_0,y_1,z_1)) \geq \inf|g_z|h $ and
\[ |h'(1)| = |D_iu (x_1-x_0)_i| = \left|D_iv \frac{(x_1-x_0)_i}{r_i}\right|  \leq |Dv|, \]
 where we've used that by the minimum ellipsoid $|(x_1-x_0)_i|/r_i \leq 2$. Thus \eqref{eq:h} becomes $|Dv| \geq ch$. To conclude we control $|Dv|$ by $|D_nv|$. Indeed, by \eqref{eq:orthog-reln}  $D_1v + \rho_1D_nv = 0$, similar reasoning implies $D_iv+\rho_iD_nv$ for $i =2,3,\dots,n-1$. Thus
  \[ Dv = (-\rho_1D_nv,\dots,-\rho_{n-1}D_{n}v,D_nv).\]
  Recalling $|D\rho| \leq C$ we complete the proof with the observation $|Dv| \leq C |D_nv|$.
\end{proof}

\section{$C^{1,\alpha}$ regularity}
\label{sec:c1-alpha-regularity}

The $C^{1,\alpha}$ regularity is essentially an immediate consequence of Lemmas \ref{lem:sec1} and \ref{lem:lemma-shape-sections}. The proof of Theorem \ref{thm:main1} is as follows. 

\begin{proof}  \textit{Step 1. [Proof for strictly $g$-convex functions]} 
  Fix $x_0$, without loss of generality equal to $0$, and then $h$ sufficiently small to ensure $S_h(x_0) \subset \subset \Omega$.  By Lemma \ref{lem:sec1}
  \begin{equation}
    \label{eq:h-est}
     h^n \leq C|S_h|\nu(S_h). 
  \end{equation}
  Then assuming the minimum ellipsoid of $S_h(x_0)$ is given by \eqref{eq:min-ellipsoid}, \eqref{eq:h-est} becomes
  \[ h^n \leq C(r_1\dots r_n)\int_{S_h}f \ dx. \]
  Using H\"older's inequality (with the second function equal to 1) we have
  \begin{align}
\nonumber    h^n &\leq C (r_1\dots r_n)|S_h|^{1-1/p}\Vert f \Vert_{L^p}\\
\label{eq:conclude}    &\leq C (r_1\dots r_n)^{2-1/p}\Vert f \Vert_{L^p}.
  \end{align}
Now we conclude as in \cite{MR2476419} (which is where Lemma \ref{lem:lemma-shape-sections} is used). More precisely \eqref{eq:conclude} is the 5th inequality on \cite[pg. 446]{MR2476419}, so the rest of this step is exactly as given there.

\textit{Step 2. [Proof for $g$-convex functions]} When $u$ is not strictly $g$-convex, we may consider on a small enough neighbourhood of $x_0$ the function $u+ \epsilon|x-x_0|^2$. Indeed by the proof of \cite[Theorem 2.22]{RankinThesis} this function is strictly $g$-convex on a neighbourhood of $x_0$ depending only on $g$ (in particular, independent of $u$ and $\epsilon$). Moreover it is an Aleksandrov solution of a generated Jacobian equation with right-hand side in the original  $L^p$ space. This is an consequence of the identity $Yu(x) = Y(x,u,(x),\partial u(x))$ for $\partial u(x)$ the subgradient.  Thus by the previous proof
\[ 0 \leq u(x)-g(x,y_0,z_0) \leq u(x)+\epsilon|x-x_0|^2-g(x,y_0,z_0) \leq C|x-x_0|^{1+\alpha},\]
as required. 
\end{proof}

\begin{remark}
  Loeper \cite{MR2506751} and Jeong \cite{MR4265360} proved the $C^{1,\alpha}$ regularity for Aleksandrov solutions of
  \[ \det DYu = \nu,\]
  where $\nu$ satisfies that for some $p \in (n,+\infty]$ and $C_\nu>0$ there holds
  \[ \nu(B_\epsilon(x)) \leq C_\nu\epsilon^{n((1-1/p))}.\]
  Our proof is easily adapted to this condition. In \eqref{eq:h-est} we use
  \[ \nu(S_h) \leq \nu(B_{2r_1}(x_0)) \leq Cr_1^{n((1-1/p))}, \]
  and combine with Lemma \ref{lem:lemma-shape-sections}. 
  
\end{remark}

\section{Interior Pogorelov estimate for constant right-hand side}
\label{sec:inter-pogor-estim}

Now we start work on the $C^{2,\alpha}$ estimate. Recall we will prove this by establishing a $C^2$ (i.e. uniform ellipticity) estimate when the right-hand side is Dini continuous. Then we obtain the $C^{2,\alpha}$ estimate from the elliptic theory \cite{MR2273802}. The $C^2$ estimate for Dini continuous right hand side is a perturbation of the same result for the special case of constant right-hand side, the proof of which is the goal of this section. First we introduce a strengthening of the $C^{1,\alpha}$ result and a strict $g$-convexity estimate that holds by duality.

\begin{lemma}
  Let $g$ be a generating function satisfying A3 and $u:\Omega\rightarrow \mathbf{R}$ be a $g$-convex solution of $\lambda \leq det DYu \leq \Lambda$. For each $\Omega' \subset \subset \Omega$ there is $C,d,\beta,\gamma > 0$ depending on $\lambda,\Lambda,g,\Omega',\Omega$ for which the following holds. Whenever $x_0 \in \Omega'$,  $g(\cdot,y_0,z_0)$ is the $g$-support at $x_0$ and $x \in B_d(x_0)$ there holds
    \begin{equation}
    \label{eq:scd}
       C^{-1}|x-x_0|^{1+\gamma} \leq u(x) - g(x,y_0,z_0) \leq C|x-x_0|^{1+\beta}.
  \end{equation}
\end{lemma}
The right-hand inequality is the $C^{1,\alpha}$ estimate of Guillen and Kitagawa \cite{MR3639322} which follows from the strict convexity in \cite[Lemma 4.1]{MR4218157}. The left-hand side inequality follows from the right-hand side by duality. We give the proof in Appendix \ref{sec:quant-conv-via}. 

\begin{lemma}\label{lem:int-pogorelov}
    Assume that $u \in C^4(\Omega) \cap C^2(\overline{\Omega})$ is a $g$-convex solution of
    \begin{align}
      \det[D^2u-A(\cdot,u,Du)] = f_0 \text{ in }\Omega, \nonumber\\
      u = g(\cdot,y,z-h) \text{ on }\partial \Omega,
    \end{align}
where $g(\cdot,y,z)$ is a support at some $\overline{x} \in \Omega$ and $f_0 > 0$ is a real number. We assume $g$ satisfies A3w, $h>0$ is sufficiently small (determined in the proof), and $C_0$ is the good shape constant of $\Omega$.  
For each $\tau \in (0,1)$ there is  $C > 0$ depending on $\Vert A \Vert_{C^2},\Vert u \Vert_{C^1(\Omega)},\tau,C_0,h$ such that in
\[ S_{\tau h} := \{ x ; u(x) < g(x,y,z-\tau h)\},\] we have
  \begin{equation}
    \label{eq:r2:pogloc}
       \sup_{S_{\tau h}}|D^2u| \leq C.
  \end{equation}
\end{lemma}
\begin{proof}
This is essentially the estimate
  \begin{equation}
    \label{eq:init-est}
     (g(\cdot,y,z-h)-u(\cdot))^\beta|D^2u(\cdot)| \leq C,
  \end{equation}
  which was given in \cite{MR4218157}. Since we need to ensure the constant only depends on $\Vert A \Vert_{C^2}$ we will provide full details. The proof is via a Pogorelov type estimate: We consider a certain test function which attains a maximum in $\Omega$. Our choice of test function ensures it both controls the second derivatives and is controlled at its maximum point.

  We let $\phi := |x-\overline{x}|^2$ and introduce both the function
\[ v(x,\xi) = \kappa\phi+\tau|Du|^2/2 + \log(w_{\xi\xi}) + \beta \log[g(\cdot,y,z -h)-u],\]
  and the differential operator\footnote{Here differentiation is with respect to $x$, never $\xi$.}
  \begin{equation}
  \label{eq:l-def}
  L(v) := w^{ij}[D_{ij}v-D_{p_k}A_{ij}D_kv],
\end{equation}
where $w_{ij} = u_{ij}-A_{ij}(x,u,Du)$ and $D_{p_k}A_{ij}$ is evaluated at $(x,u,Du)$. 

We use the notation $u_0 = g(\cdot,y,z-h)$ and $\eta = u_0-u$. Because the nonnegative function $e^v$ is $0$ on $\partial \Omega$,  $v$ attains an interior maximum at $x_0 \in \Omega$ and some $\xi$ assumed without loss of generality to be $e_1$. We also assume, again without loss of generality, that at $x_0$ $w$ is diagonal.  At an interior maximum $Dv=0$ and $D^2v \leq 0$ so that
  \begin{equation}
    \label{eq:p:globmain}
       0 \geq L v  = \kappa L\phi + \tau L(|Du|^2/2) + L(\log(w_{11})) + \beta L \log \eta.
  \end{equation}
  We will compute  each term in \eqref{eq:p:globmain} and from this obtain \eqref{eq:init-est}.

  \textit{Term 1: $L\phi$.} This one's immediate -- provided the domain (and subsequently $|D\phi|$) is chosen sufficiently small depending on $|A_{ij,p}|$ we have
  \begin{equation}
    \label{eq:p:globest1}
     L \phi \geq w^{ii}-C.
  \end{equation}

  \noindent \textit{Term 2: $L(|Du|^2)$}
  We compute
  \begin{align}
 \label{eq:p:2terms}  L(|Du|^2/2)  &= w^{ii}u_{ki}u_{ki} + u_k[w^{ii}(u_{kii}- D_{p_l}A_{ii}u_{lk})].
  \end{align}
  We note (using $u_{ij} = w_{ij}+A_{ij}$) that
  \begin{equation}
   w^{ii}u_{ki}u_{ki} = w^{ii}(w_{ki}+A_{ki})(w_{ki}+A_{ki}) \geq w_{ii}-C(1+w^{ii}),\label{eq:r2:wuu}
 \end{equation}
  and by differentiating the PDE in the direction $e_k$ at $x_0$
  \begin{align}
 \label{eq:r2:diff-once}   w^{ij}[u_{ijk}-A_{ij,p_l}u_{lk}] = w^{ij}(A_{ij,k}+A_{ij,u}u_k).
  \end{align}
  Hence \eqref{eq:p:2terms} becomes
  \begin{equation}
  L(|Du|^2/2) \geq w_{ii}-C(1+w^{ii}).\label{eq:p:globest2}
\end{equation}

\noindent \textit{Term 3: $L(\log(w_{11}))$.} To begin, we differentiate the PDE twice in the $e_1$ direction and obtain, with the notation $w_{ij,k} := D_{x_k}w_{ij}$, that
\begin{align}
  \nonumber w^{ii}[u_{ii11}-D_{p_k}A_{ii}u_{k11}]&= w^{ii}w^{jj}w_{ij,1}^2 + w^{ii}\big[A_{ii,11} - 2A_{ii,1u}u_1\\
        \nonumber          +2A_{ii,1p_k}u_{k1}+A_{ii,uu}&u_1^2+A_{ii,u}u_{11}+2A_{ii,p_k}u_1u_{1k}+A_{ii,p_kp_l}u_{1k}u_{1l} \big]\\
        \geq w^{ii}w^{jj}w_{ij,1}^2&+w^{ii}A_{ii,p_kp_l}u_{1k}u_{1l} - C(w_{ii}+w^{ii}+w_{ii}w^{ii}). \label{eq:p:1}
\end{align}
We use A3w to deal with the second term. Use $u_{ij} = w_{ij}+A_{ij}$ to write
\[ w^{ii}A_{ii,p_kp_l}u_{1k}u_{1l} \geq w^{ii}A_{ii,p_1p_1}w_{11}^2 - C(w^{ii}+w^{ii}w_{ii}).\]
Then by applying A3w with $\xi=e_i,\ \eta=e_j$ for $i \neq j$ we see $A_{ii,p_jp_j} \geq 0$ so that
\begin{align*}
  \label{eq:12}
  w^{ii}A_{ii,p_1p_1}w_{11}^2 &= w^{11}A_{11,p_1p_1}w_{11}^2 + \sum_{i=2}^nw^{ii}A_{ii,p_1p_1}w_{11}^2  \\
  &\geq -Cw_{11}.
\end{align*}
Thus \eqref{eq:p:1} becomes
\[ L(u_{11}) \geq w^{ii}w^{jj}w_{ij,1}^2- C(w_{ii}+w^{ii}+w_{ii}w^{ii}). \]

We perform similar calculations for $LA_{11}$ to obtain 
\begin{equation}
  \label{eq:p:lw11}
   L(w_{11}) \geq w^{ii}w^{jj}w_{ij,1}^2- C(w_{ii}+w^{ii}+w_{ii}w^{ii}).
\end{equation}
We use this to compute $L(\log w_{11})$ as follows. First,
\[ L(\log w_{11})  = -\frac{w^{ii}w_{11,i}^2}{w_{11}^2} + \frac{L(w_{11})}{w_{11}}.\]
Hence by \eqref{eq:p:lw11}
\begin{equation}
L(\log w_{11}) \geq \frac{w^{ii}w^{jj}w_{ij,1}^2}{w_{11}}-\frac{w^{ii}w_{11,i}^2}{w_{11}^2} - \frac{C}{w_{11}}(w_{ii}+w^{ii}+w_{ii}w^{ii}).\label{eq:p:toreturn}
\end{equation}
When $i,j = 1$ in the first term and $i=1$ in the second these terms cancel. At the expense of an inequality we discard terms with neither $i$ nor $j = 1$. Thus
\begin{align}
  \label{eq:r2:pog1}   \frac{w^{ii}w^{jj}w_{ij,1}^2}{w_{11}}&-\frac{w^{ii}w_{11,i}^2}{w_{11}^2}\\
\nonumber  &\geq \sum_{i>1}\frac{w^{ii}w_{i1,1}^2}{w_{11}^2} + \sum_{j>1}\frac{w^{jj}w_{1j,1}^2}{w_{11}^2} - \sum_{i>1}\frac{w^{ii}w_{11,i}^2}{w_{11}^2}\\
  \nonumber  &=\frac{1}{w_{11}^2}\sum_{i>1}w^{ii}\big[2w_{i1,1}^2-w_{11,i}^2\big]\\
  \nonumber  &= \frac{1}{w_{11}^2}\sum_{i>1}w^{ii}w_{11,i}^2 + \frac{2}{w_{11}^2}\sum_{i>1}w^{ii}[w_{i1,1}^2 - w_{11,i}^2]\\
  \nonumber         &= \frac{1}{w_{11}^2} \sum_{i>1}w^{ii}w_{11,i}^2\\
\nonumber  &\quad\quad+ \frac{2}{w_{11}^2}\sum_{i>1}w^{ii}(w_{i1,1}-w_{11,i})(w_{i1,1}+w_{11,i}).
\end{align}
Rewriting the second sum in terms of the $A$ matrix yields
\begin{align*}
  &\frac{w^{ii}w^{jj}w_{ij,1}^2}{w_{11}}-\frac{w^{ii}w_{11,i}^2}{w_{11}^2}= \frac{1}{w_{11}^2} \sum_{i>1}w^{ii}w_{11,i}^2\\
  &\quad\quad+ \frac{2}{w_{11}^2}\sum_{i>1}w^{ii}(D_iA_{11}-D_{1}A_{i1})(2w_{11,i}+D_iA_{11}-D_1A_{i1})\\
    &= \frac{1}{w_{11}^2} \sum_{i>1}w^{ii}\big[w_{11,i}^2 + 4w_{11,i}(D_iA_{11}-D_{1}A_{i1}) + 2(D_iA_{11}-D_{1}A_{i1})^2\big].
\end{align*}
 Now, Cauchy's inequality implies
\begin{align*}
  4w_{11,i}(D_iA_{11}-D_{1}A_{i1}) \geq -\frac{w_{11,i}^2}{2} - 8 (D_iA_{11}-D_{1}A_{i1})^2.
\end{align*}
Thus
\[  \frac{w^{ii}w^{jj}w_{ij,1}^2}{w_{11}}-\frac{w^{ii}w_{11,i}^2}{w_{11}^2} \geq \frac{1}{2w_{11}^2}\sum_{i=2}^nw^{ii}w_{11,i}^2 - Cw^{ii},\]
and on returning to \eqref{eq:p:toreturn}
\begin{equation}
  \label{eq:p:globest3}
  L(\log(w_{11})) \geq \frac{1}{2w_{11}^2}\sum_{i=2}^nw^{ii}w_{11,i}^2 - C(1+w_{ii}+w^{ii}).   
\end{equation}
Now, using \eqref{eq:p:globest1},\eqref{eq:p:globest2}, and \eqref{eq:p:globest3} in  \eqref{eq:p:globmain} implies
\begin{align}
  \label{eq:r2:pogloc1} 0 &\geq \kappa(w^{ii}-C) + \tau [w_{ii}-C(1+w^{ii})] +  \frac{1}{2w_{11}^2}\sum_{i=2}^nw^{ii}w_{11,i}^2 \\
\nonumber  & \quad\quad - C(1+w_{ii}+w^{ii})+ \beta L(\log \eta).
\end{align}

\noindent \textit{Term 3: $L(\log \eta)$.} First, write
\begin{equation}
    \label{eq:r2:pogloc2} L(\log \eta) = \frac{L\eta}{\eta} - \sum_{i=1}^nw^{ii}\left(\frac{D_i\eta}{\eta}\right)^2. 
\end{equation}
We compute (using $D_{ii}u_0 = A_{ii}(\cdot,u_0,Du_0)$)
\begin{align}
\nonumber  L\eta = w^{ii}&[D_{ii}u_0 - D_{ii}u - D_{p_k}A_{ii}(\cdot,u,Du)D_k\eta] \\
 \nonumber    =w^{ii}&[-w_{ii}+A_{ii}(\cdot,u_0,Du_0)- A_{ii}(\cdot,u,Du)- D_{p_k}A_{ii}(\cdot,u,Du)D_k\eta] \\
  \nonumber   \geq w^{ii}&[A_{ii,u}\eta+A_{ii}(\cdot,u,Du_0)- A_{ii}(\cdot,u,Du)- D_{p_k}A_{ii}(\cdot,u,Du)D_k\eta] - C\\
  &\geq w^{ii}D_{p_kp_l}A_{ii}D_k\eta D_l \eta - C - Cw^{ii}\eta. \label{eq:r2:pogloc3}
\end{align}
For each $i$ write
\begin{align*}
  D_{p_kp_l}A_{ii}D_k\eta D_l \eta &= \sum_{k,l \neq i}D_{p_kp_l}A_{ii}D_k\eta D_l \eta + 2\sum_{l \neq i}D_{p_ip_l}A_{ii}D_i\eta D_l \eta \\
  &\quad\quad+ D_{p_ip_i}A_{ii}D_i\eta D_i \eta
\end{align*}
Then by A3w the first term is nonnegative, so that
\begin{align*}
D_{p_kp_l}A_{ii}D_k\eta D_l \eta  &\geq -C D_i\eta - C (D_i\eta)^2.
\end{align*}
Returning to \eqref{eq:r2:pogloc3} we see
\[  L\eta \geq -C(1+w^{ii}\eta) - Cw^{ii}D_i\eta - Cw^{ii}(D_i\eta)^2.\]
Which into \eqref{eq:r2:pogloc2} implies
\begin{equation}
  \label{eq:r2:pogloc4}
   L(\log \eta) \geq -\frac{C}{\eta} - Cw^{ii} - C \sum_{i=1}^nw^{ii}\left(\frac{D_i\eta}{\eta}\right)^2.
\end{equation}
Here we've used that we can assume $\eta < 1$, and also used Cauchy's to note
\[ w^{ii}\frac{D_i\eta}{\eta} = \sqrt{w^{ii}}\sqrt{w^{ii}}\frac{D_i\eta}{\eta} \leq w^{ii}+w^{ii}\left(\frac{D_i\eta}{\eta}\right)^2. \]
Now we deal with the final term in \eqref{eq:r2:pogloc4}. We can assume that  $w^{11}(D_1\eta/\eta)^2 \leq 1$, for if not we have \eqref{eq:r2:pogloc} with $\beta = 2$. Since we are at a maximum $D_iv = 0$, that is
\[ \frac{D_i\eta}{\eta} = -\frac{1}{\beta}\left[\frac{w_{11,i}}{w_{11}}+\kappa D_i\phi + \tau D_k u w_{ik} + \tau D_k u A_{ik}\right].\]
This implies
\[ \sum_{i=2}^nw^{ii}\left(\frac{D_i\eta}{\eta}\right)^2 \leq C\left[\sum_{i=2}^n\frac{w^{ii}w_{11,i}^2}{w_{11}^2\beta^2} + \frac{\kappa^2}{\beta^2}w^{ii}|D_i\phi|^2 + \frac{\tau^2}{\beta^2}(w_{ii}+w^{ii})\right].\]
Choosing $\beta \geq 1,2C$ and returning to \eqref{eq:r2:pogloc4} we obtain
\[ \beta L(\log \eta) \geq \frac{-C\beta}{\eta}-\kappa^2|D\phi|^2 w^{ii} - \frac{C\tau^2}{\beta^2}[w^{ii}+w_{ii}] - C\beta w^{ii} - \frac{1}{2w_{11}^2}\sum_{i=2}^nw^{ii}w_{11,i}^2.\]

Substituting into \eqref{eq:r2:pogloc1} completes the proof: 
\begin{align*}
  0 &\geq \kappa (w^{ii}-C)+\tau[w_{ii}-C(1+w^{ii})] - C(1+w_{ii}+w^{ii}) -\frac{C\beta}{\eta} \\
    &\quad\quad-C \kappa^2|D\phi|^2 w^{ii} - C\frac{\tau^2}{\beta}[w^{ii}+w_{ii}] - C\beta w^{ii}\\
    & = w^{ii}[\kappa - \tau C -C-C\kappa^2|D\phi|^2 -\frac{C\tau^2}{\beta} - C\beta]\\
  &\quad\quad+w_{ii}[\tau - C - \frac{C\tau^2}{\beta}] - C[\kappa+\tau+\frac{\beta}{\eta}].
\end{align*}
Take $\text{diam}(\Omega)$, and subsequently $|D\phi|$, small enough to ensure $\kappa^2|D\phi|^2 \leq 1$ (our choice of $\kappa$ will only depend on allowed quantities). A further choice of $\beta \geq \tau^2$, $\tau$ large depending only on $C$, and finally $\kappa$ large depending on $\tau,C,\beta$ implies
\[ 0 \geq w^{ii}+w_{ii}-\frac{C}{\eta}.\]
This implies $\eta w_{ii} \leq C$ at the maximum point, and the proof is complete. 
\end{proof}

\section{$C^{2,\alpha}$ regularity}
\label{sec:c2-alpha-regularity}

In this section we will prove the $C^{2,\alpha}$ estimate via a $C^2$ estimate. We adapt the method of proof used by Liu, Trudinger, and Wang in the optimal transport case \cite{MR2748621}. We also use some details from Figalli's exposition in the Monge--Amp\`ere case \cite[Section 4.10]{MR3617963}.  Here's how we obtain the $C^2$ estimate. When the right hand side is constant the $C^2$ estimate is true by the interior Pogorelov's estimate of the previous section. Then the argument for Dini-continuous $f$ is to perturb the argument for constant $f$. That is we zoom in, treating a series of normalized approximating problems with constant right hand side.

\subsection{Normalization of sections}
\label{sec:norm-sect}
Here we explain the procedure for normalizing a solution on a section. We assume we are given a strictly $g$-convex function $u:\Omega \rightarrow \mathbf{R}$ which is an Aleksandrov solution of
\begin{align*}
  \det DYu = f,
\end{align*}
as well as a point $x_0$ and corresponding $g$-support $g(\cdot,y_0,z_0)$. As usual, we consider the section $S_h(x_0)$. The definition of Aleksandrov solution is coordinate independent so we may assume we are in the coordinates given by \eqref{eq:x-trans} and $S_h$ is convex. Assume the minimum ellipsoid is given by \eqref{eq:min-ellipsoid} and $T$ is given by \eqref{eq:t-def}.  We want to consider the PDE solved by
\begin{equation}
  \label{eq:v-def}
   v(\overline{x}) = \frac{1}{h}[u(T^{-1}\overline{x}) - g(T^{-1}\overline{x},y_0,z_0-h)].
\end{equation}
on $U := TS_h$. Importantly $B_{1/n} \subset U \subset B_1$, $v|_{\partial U} = 0$, and $C^{-1} \leq |\inf_{U} v| \leq C $ for $C$ depending only on $g_z$. Thus this is a natural generalization of the normalization procedure for Monge--Amp\`ere equations. We show that $v$ solves a MATE
  \begin{align}\label{eq:v-eq}
       \det [D^2v - \overline{A}(\cdot,v,Dv)] = \overline{B}(\cdot,v,Dv),
  \end{align}
  for $\overline{A},\overline{B}$ satisfying $\Vert \overline{A} \Vert_{C^2} \leq C\Vert A \Vert_{C^2}$ and $C^{-1} B \leq \overline{B} \leq C B$ for $C$ depending only on $g$. When $v$ is defined by \eqref{eq:v-def} we use the notation
\begin{align*}
  \tilde{A}_{ij}(\overline{x},v,Dv) &= \frac{r_ir_j}{h}A_{ij}\left(T^{-1}\overline{x},u,Du\right), \\
 Yv &= Y\left(T^{-1}x,u,Du\right)
\end{align*}
and similarly for $Zv$. We compute directly $v$ solves the equation \eqref{eq:v-eq} for 
\begin{align*}
  \overline{A}_{ij}(\overline{x},v,Dv) &= \frac{r_ir_j}h\left[g_{ij}(T^{-1}\overline{x}, Yv, Zv) - g_{ij}(T^{-1}\overline{x}, y_0, z_0 - h)\right]\\
                                       &= \frac{r_ir_j}h\left[g_{ij}(T^{-1}\overline{x}, Yv, Zv) - g_{ij}(T^{-1}\overline{x}, y_0, z_0 ) + O(h)\right]\\
                                       &= \tilde{A}_{ij}(\overline{x},v(\overline{x}),Dv(\overline{x})) - \tilde{A}_{ij}(\overline{x},v(\overline{x}_0),Dv(\overline{x}_0)) + O(h)\\
  &= \tilde{A}_{ij}(\overline{x},v(\overline{x}),Dv(\overline{x})) - \tilde{A}_{ij}(\overline{x},v(\overline{x}),Dv(\overline{x}_0)) + O(h). 
\end{align*}
In this form we can argue exactly as in \cite[pg. 440]{MR2476419} to obtain
\[   \overline{A}_{ij}(\overline{x},v,Dv) =\frac{hr_ir_j}{r_kr_l}A_{ij,p_kp_l}(T^{-1}\overline{x},hv,\frac{h}{r_i}v_i)v_kv_l + O(h). \]
Then $\overline{A}$ is bounded by Lemma \ref{lem:lemma-shape-sections} when A3 is satisfied. For the $C^2$ estimates for $\overline{A}$ note $\frac{h}{r_i}$ is bounded by the strict convexity and differentiability estimate \eqref{eq:scd}. Finally the pinching estimate on $B$ follows from $\overline{B}(\cdot,v,Dv) = h^n( \det T)^2 B(\cdot,u,Du)$, $(\det T)^{-1} = r_1 \dots r_n$, and \eqref{eq:sec-ellips}.

\subsection{Lemmas for the $C^{2,\alpha}$ estimate}
\label{sec:c2-estim-const}

Here we show we can estimate the good shape constant of sections by the second derivatives of the function and vice versa. A subtlety is that sections of different height are convex in different coordinates. We assume at the outset some initial coordinates are fixed. We will say a section has good shape constant $C$ if after performing the change of variables \eqref{eq:x-trans} the section has good shape constant $C$. Note because the Jacobian of the transformation \eqref{eq:x-trans} is $E$ (from A2) if $S_h$ has good shape constant $C$ then it is still true, in the initial coordinates with respect to which $S_h$ may not be constant, that $S_h$ contains a ball of radius $r$ and is contained in a ball of radius $R$ for $R/r \leq C \frac{|\Lambda|}{|\lambda|}$ where $\Lambda,\lambda$ are respectively the minimum and maximum eigenvalues of $E$ over $\overline{\Gamma}$.

\begin{lemma}\label{lem:deriv-imp-shape}
Let $u$ be a strictly $g$-convex solution of
  \begin{align}
    \det [D^2u - A(\cdot,u,Du)] = f \text{ in }\Omega.
  \end{align}
  Fix $x_0 \in \Omega$ and a support $g(\cdot,y_0,z_0)$ at $x_0$. Assume $S_h(x_0) \subset \subset \Omega$, and $|D^2u(x_0)| \leq M$ for some $M$. Then $S_h(x_0)$ has a good shape constant which depends only on $M,h,f,g$ and the constant in the Pogorelov Lemma \ref{lem:int-pogorelov}. 
\end{lemma}
\begin{proof}
  We normalize the section $S_h$ and solution as in Section \ref{sec:norm-sect} and let the normalized solution be denoted by $v$. By the Pogorelov interior estimates we have
  \[cI \leq \overline{w}(Tx_0) \leq CI,\]
  for $\overline{w} = D^2v - \overline{A}(x,v,Dv)$. 
Similarly by the PDE and $|D^2u(x_0)| \leq M$ we have $cI \leq w(x_0) \leq CI$ for $w(x_0) = D^2u(x_0) - A(x_0,u(x_0),Du(x_0))$. 
  Using in addition
  \[ \frac{1}{h}w(x_0) = T^T \overline{w}(Tx_0) T,\]
  we obtain (see \cite[Eq. 4.26]{MR3617963}) a bi-Lipschitz estimate $\Vert T \Vert, \Vert T^{-1} \Vert \leq C$. Since we can assume $T$ is of the form \eqref{eq:t-def} we obtain $r_1/r_n \leq C^2$, the desired good shape estimate on $S_h(x_0)$. 
\end{proof}

\begin{lemma}\label{lem:shape-imp-deriv}
  Assume $u \in C^3(\Omega)$ is a strictly $g$-convex  solution of 
  \[ \det[D^2u - A(\cdot,u,Du)] = f \text{ in }\Omega.\]
 Fix $x_0$ and assume $g(\cdot,y_0,z_0)$ is a $g$-support at $x_0$. If there is a sequence $h_k \rightarrow 0$ such that each
  \[ S_{h_k}:= \{ u < g(\cdot,y_0,z_0-h_k)\},\]
  has a good shape constant less than some $C_0$ then $|D^2u(x_0)| \leq C$ for $C$ depending on $C_0,\sup|f|$ and $g$. 
\end{lemma}
\begin{proof}
  Without loss of generality $x_0=0$ and the minimum ellipsoid of $S_{h_k}$ has axis $r_1^{(k)} \geq \dots \geq r_n^{(k)}$. Our assumption is $r_1^{(k)} \leq C_0 r_n^{(k)}$. Then, by \eqref{eq:sec-ellips} 
  \[C_0^{-n+1}(r_1^{(k)})^n \leq r_1^{(k)} \dots r_n^{(k)} \leq C |S_{h_k}|  \leq Ch_k^{n/2},  \]
  that is $r_1^{(k)} \leq C\sqrt{h_k}$. Moreover, because $r_1^{(k)}$ is the largest axis of the minimum ellipsoid $S_{h_k} \subset B_{2r_1^{(k)}}(0)$ so that
  \begin{equation}
    \label{eq:sec-contain}
      S_{h_k} \subset B_{C \sqrt{h_k}}(0). 
  \end{equation}
  
  Let $w_{ij} = u_{ij}(0)-g_{ij}(0,y_0,z_0)$ and denote the minimum eigenvalue of $w$ by $\lambda$ and corresponding normalized eigenvector by $x_\lambda$. Using a Taylor series we have
  \begin{align*}
    u(tx_\lambda) - g(t x_\lambda,y_0,z_0 - h) \leq t^2\lambda+O(t^3) - \inf |g_z|h_k.
  \end{align*}
Thus provided $k$ is taken sufficiently large $tx_\lambda \in S_{h_k}$ for $t = \sqrt{\inf|g_z|\frac{h_k}{2\lambda}}$. That is, there is $x \in S_{h_k}$ with 
  \begin{align}
    |x| &\geq C\sqrt{\frac{h_k}{\lambda}}.\label{eq:contained}
  \end{align}
 Combining \eqref{eq:sec-contain} and \eqref{eq:contained} we obtain a lower bound on $\lambda$ which implies an upper bound on the largest eigenvalue of $w$ by the PDE. 
\end{proof}

Before proving the $C^2$ estimate we state one final lemma. 
\begin{lemma}\label{lem:int-schauder}
  Assume $u_i$ for $i=1,2$ is a $C^{2,\alpha}$ solution of
  \begin{align*}
    \det [D^2u_i - A(\cdot,u_i,Du_i)] = f_i \text{ in }\Omega
  \end{align*}
  where $f_i$ is a H\"older continuous function and $\Vert u_i \Vert_{C^4} \leq K.$ For any $\Omega' \subset\subset\Omega$ there is $C>0$ depending on $K,g,\text{dist}(\partial \Omega,\Omega')$ such that
  \[ \Vert u_1-u_2 \Vert_{C^3(\Omega')} \leq C \Vert f_1-f_2\Vert_{C^\alpha(\Omega)} +\sup_{\Omega}|u_1-u_2|. \]
\end{lemma}
\begin{proof}
  We linearise the PDE (as in \cite[Lemma 4.2]{MR2748621}), obtaining $L(u_1-u_2) = f_1-f_2 $ for a suitable linear operator with coefficients and ellipticity constants controlled by estimates $\Vert u_i \Vert_{C^4} \leq K.$ The lemma then follows from the classical Schauder theory \cite[Theorem 6.2]{MR1814364}.
\end{proof}

\subsection{Proof of the  $C^{2,\alpha}$ estimate}
\label{sec:c2-alpha-estimate}

As stated at the start of this section Theorem \ref{thm:main2} follows from an interior $C^2$ estimate. This is what we now prove.

\begin{theorem}
  Let $g$ be a generating function satisfying A3 and $u\in C^2(\Omega)$ be a $g$-convex solution of
  \[ \det [D^2u - A(\cdot,u,Du)] = f \text{ in }\Omega.\]
  If $f$ is Dini-continuous with $0 < \lambda \leq f \leq \Lambda< \infty$ then for each $\Omega' \subset\subset \Omega$ we have an estimate $\sup_{\Omega'}|D^2u| \leq C(f,g,\Omega,\Omega')$. 
  \end{theorem}
\begin{proof}
\textit{Step 1. [Setup of approximating problems]}  At the outset we fix $x_0 \in \Omega$, where without loss of generality  $x_0 = 0$. Consider
  \[S_h = \{ x; u(x) < g(x,y_0,z_0-h)\},\]
  where $g(\cdot,y_0,z_0)$ is the support at $0$ and $h>0$ is chosen small enough to ensure Lemma \ref{lem:int-pogorelov} applies. We normalize so that $B_{1/n} \subset S_h \subset B_1$ with $B_1$ the minimum ellipsoid. Moreover, we assume $h$ is chosen small enough to ensure that after this normalization
  \begin{equation}
    \label{eq:eps-def}
     \int_0^1\frac{\omega(r)}{r} < \epsilon,
  \end{equation}
  for an $\epsilon$ to be chosen in the proof (recall $\omega$ is from Definition \ref{defn:dini}). Note such a choice of $h$ is controlled by \eqref{eq:scd} and thus up to rescaling we assume $h=1$. 

W introduce a sequence of approximating problems. Define the domains
  \[ U_k = \{x; u(x) < g(x,y_0,z_0-1/4^k)\},\]
  and let $f_k = \inf_{U_k}f$.

Let  $u_k \in C^4(U_k)$ be the solution (whose existence is guaranteed by the Perron method) of 
  \begin{align}
   \label{eq:approx-prob} \det[D^2u_k-A(\cdot,u_k,Du_k)] = f_k \text{ in } U_k,\\
\label{eq:approx-bc}    u_k = u \text{ on }\partial U_k.
  \end{align}
In addition put 
\begin{align*}
  \nu_k = \sup_{x,y \in U_k} |f(x)- f(y)|.
\end{align*}
Using \eqref{eq:sec-ellips} the section  $U_k$ is contained in a ball of of radius $C2^{-k}$ for $C$ depending on the good shape constant of $U_k$. Thus when the good shape constant of $U_k$ is controlled $C \nu_k \leq \omega(2^{-k})=:\omega_k$ which we'll use at the conclusion of the proof. Lemmas \ref{lem:deriv-imp-shape} and \ref{lem:shape-imp-deriv} suggest the structure of our proof: It suffices to show a good shape constant of each $U_k$ is controlled by a fixed constant independent of $k$ and this, in turn, follows from a uniform estimate on each $|D^2u_k|$ inside a subsection. More precisely, we'll prove by induction that for each $k=0,1,2,\dots$ we have
\begin{equation}
  \label{eq:ind-hyp}
    |D^2u_k(x)| \leq \sup_{V^{\tau_0}_0}|D^2u_0| + 1 =: M,
 \end{equation}
 for any $x \in V^{\tau_0}_k$ which is defined by
 \[ V^{\tau}_k = \{x ; u_k(x) < g(x,y_0,z_0-\frac{1-\tau}{4^k})\}. \]
 By the uniform estimates in \cite{Rankin2021}\footnote{Use the upper bound \cite[Theorem 4]{Rankin2021} and corresponding lower bound obtained as in \cite[Theorem 6.2]{MR3067826}. } there is a  choice of $\tau_0$ sufficiently close to $0$ such that 
 \begin{align}
   \label{eq:x0-contain} x_0 &\in V^{\tau_0}_k,\\
   \label{eq:vtau-contain} U_{k+1} &\subset V^{\tau_0}_k.
 \end{align}

\textit{Step 2. [Induction base case: $k=0$]} It is clear that  \eqref{eq:ind-hyp} holds for $k=0$. However we note here that $M$ is controlled by the interior Pogorelov estimate, that is in terms of $\tau_0,h$ and our initial normalizing transformation which is, in turn, controlled by \eqref{eq:sec-ellips}.  

\textit{Step 3. [Inductive step]}
Now  we assume \eqref{eq:ind-hyp} up to some fixed $k$. We rescale our solution and domain by introducing
\begin{align*}
  &\overline{u}_k(\overline{x}):= 4^ku_k\left(\frac{\overline{x}}{2^k}\right) &&\overline{u}_{k+1}(\overline{x}):= 4^ku_{k+1}\left(\frac{\overline{x}}{2^k}\right),
\end{align*}
where $\overline{x} = 2^kx$. The function $\overline{u}_k$ solves
\begin{align*}
  &\det[D^2\overline{u}_k - \overline{A}(\cdot,\overline{u}_k,D\overline{u}_k)] = f_k \text{ in }\overline{U}_k\\
 & \overline{u}_k = 4^{-k}g(\cdot,y_0,z_0-h/4^k) \text{ on }\partial \overline{U}_k,
\end{align*}
for $\overline{U}_k := 2^kU_k$ and $\overline{A}(x,u,p) = A(2^{-k}x,4^{-k}u,2^{-k}p)$ and similarly for $\overline{u}_{k+1}$. Note this transformation does not change the magnitude of the second derivatives. Thus the inductive hypothesis \eqref{eq:ind-hyp} and Lemma \ref{lem:deriv-imp-shape} implies $\overline{U}_k$ has a good shape constant depending only on $M,f$ and the constant in the Pogorelov lemma. We claim that $\overline{U}_{k+1}$ has a good shape constant depending on the same parameters and the constants in \eqref{eq:sec-ellips}.  To see this assume the minimum ellipsoids of $\overline{U}_{k}$ and $\overline{U}_{k+1}$ have axis $R_1 \geq \dots \geq R_n$ and $r_1 \geq \dots \geq r_n$ respectively. By \eqref{eq:sec-ellips} applied to the section $\overline{U}_{k+1}$ we have
\[ C4^{-n(k+1)/2} \leq C |\overline{U}_{k+1}| \leq  r_1 \dots r_n \leq r_1^{n-1} r_n.  \]
Using this to compute an upper bound on $1/r_n$, we obtain
\begin{equation}
  \label{eq:good-shape-1}
  \frac{r_1}{r_n} \leq C4^{n(k+1)/2}r_1^n \leq  C4^{n(k+1)/2}R_1^n,  
\end{equation}
where the final inequality is because $\overline{U}_k$ is the larger section. Now, let $c_0$ be the good shape constant of $\overline{U}_k$, that is $R_1 \leq c_0R_n$. Using \eqref{eq:sec-ellips} again, this time applied to the section $\overline{U}_{k+1}$, we have
\begin{equation}
  \label{eq:good-shape-2}
   C4^{-nk/2} \geq R_1 \dots R_n \geq R_1 \dots R_n^{n-1} \geq c_0^{-(n-1)}R_1^{n}.
\end{equation}
Combining \eqref{eq:good-shape-1} and \eqref{eq:good-shape-2} implies the claimed fact that when $\overline{U}_k$ has good shape so does $\overline{U}_{k+1}$\footnote{To be explicit, because we are proving \eqref{eq:ind-hyp} by induction (not a good shape estimate by induction) it does not matter that the good shape constant of $\overline{U}_{k+1}$ is worse than $\overline{U}_k$.}. 

 Noting that $\Vert \overline{A} \Vert_{C^2} \leq \Vert A \Vert_{C^2}$ we obtain by Lemma \ref{lem:int-pogorelov} a $C^2$ estimate depending on allowed quantities and $M$ in $\overline{V}^{\tau_0/4}_k$.  Then the Evans--Krylov interior estimates imply $C^{2,\alpha}$ estimates in $\overline{V}^{\tau_0/2}_k$ and subsequently higher estimates by the elliptic theory. Similarly for $\overline{u}_{k+1}$ in the corresponding sections $\overline{V}^{\tau_0/4}_{k+1},\overline{V}^{\tau_0/2}_{k+1}$ .

Linearising and using the maximum principle on small domains(see \cite[Lemma A.3]{RankinThesis}) we obtain
\[ |\overline{u}_k-\overline{u}|,|\overline{u}_{k+1}-\overline{u}| \leq C\nu_k.\]
Thus $|\overline{u}_k - \overline{u}_{k+1}| \leq C\nu_k$ in $\overline{V}^{\tau_0/2}_{k+1}$. Then,  by Lemma \ref{lem:int-schauder}, in $V^{\tau_0}_{k+1}$ we have
\[ |D^2u_k(x) - D^2u_{k+1}(x)| \leq C \nu_k. \]
(Note we first obtain this for $\overline{u}_k,\overline{u}_{k+1}$ then use $D^2\overline{u}_k = D^2u_k$.) We've used that by \eqref{eq:vtau-contain} the estimates for $\overline{u}_k$ in $\overline{V}^{\tau_0}_{k}$ hold on the entire smaller section $\overline{U}_{k+1}$. 
Since all we've used is the induction hypothesis we can conclude the same inequality for $k$ replaced by $i = 0,1,2,\dots,k$. Moreover because these sections have a controlled good shape constant we have
\begin{equation}
  \label{eq:omega-est}
   | D^2u_i(x) - D^2u_{i+1}(x)| \leq C \omega_i = C\omega(2^{-i}),
\end{equation}
for $i = 0,1,2,\dots,k$.

Now, using \eqref{eq:omega-est}, the calculations are standard:
\begin{align*}
  \vert D^2u_{k+1}(x) \vert &\leq |D^2u_0(x)| + \sum_{i=0}^{k}|D^2u_{i}(x) - D^2u_{i+1}(x)|\\
                &\leq |D^2u_0(x)| +  C\sum_{i=0}^{k}2^{-i}\frac{\omega(2^{-i})}{2^{-i}}\\
                &\leq |D^2u_0(x)| + C\int_{2^{-k}}^1\frac{\omega(r)}{r} \ dr.
\end{align*}
Thus provided $\epsilon$ in \eqref{eq:eps-def} is taken sufficiently small we conclude by induction that  \eqref{eq:ind-hyp} holds for all $k$. By the rescaling used in the proof this implies a fixed good shape estimate for each $U_k$. Then the desired $C^2$ estimate  holds by Lemma \ref{lem:shape-imp-deriv}.
\end{proof}

\appendix

\section{Omitted calculations}
\label{sec:omitted-calculations}

\subsection{Proof that $ A_{kl,p_m}\dot{\gamma}_k\dot{\gamma}_l = 0$}
\label{sec:vanishing}

Let some initial coordinates, denoted $x$, be given and define $\overline{x} = \frac{g_y}{g_z}(x,y_0,z_h)$. For notation put $g_h(\cdot) = g(\cdot,y_0,z_h)$. We compute that for  $\xi,\eta$ in $\mathbf{R}^n$ 
\begin{align}
  \label{eq:changevar}
  \frac{\partial^2x_k}{\partial \overline{x}_i\partial \overline{x}_j}\xi_i\xi_j\eta_k &= g_z^{2}D_{p_k}A_{ij}(x,g_h(x),Dg_h(x))(E^{-1}\xi)_i(E^{-1}\xi)_j\eta_h \\
  \nonumber &\quad+ g_zg_{r,z}[E^{ir}E^{jk} + E^{ik}E^{jr}]\xi_i\xi_j\eta_k
\end{align}
Thus if our initial coordinates are the $\overline{x}$ coordinates and $\xi \cdot \eta = 0$
\begin{equation}
  \label{eq:van-term}
  D_{p_k}A_{ij}(x,g_h(x),Dg_h(x))\xi_i\xi_j\eta_k = 0.
\end{equation}
We recall from \cite[Eq. 2.3]{MR3121636} that
\begin{equation}
  \label{eq:dpgij}
   D_{p_k}g_{ij} = E^{r,k}[g_{ij,r} - g_{,r}g_{ij,z}g_z^{-1}].
\end{equation}

From the definition of $\overline{x}$ we compute
\begin{align}
\nonumber  \frac{\partial x_k}{\partial \overline{x}_j} = g_zE^{jk},\\
\label{eq:deriv-ident}  \frac{\partial }{\partial \overline{x}_i} = g_zE^{ir}\frac{\partial }{\partial x_r}.
\end{align}
 Applying \eqref{eq:deriv-ident} twice and using the identity for differentiating an inverse matrix gives
\begin{equation}
  \label{eq:dderiv-ident}
   \frac{1}{g_z}\frac{\partial^2x_k}{\partial \overline{x}_i\partial \overline{x}_j} = E^{ir}E^{jk}g_{r,z} - g_zE^{ir}E^{ja}(D_rE_{ab})E^{bk}.
\end{equation}
Now, by direct calculation
\[D_rE_{ab} = -\frac{g_{a,z}}{g_z}E_{rb} + g_{ar,b} - g_{ar,z}g_{,b}g_z^{-1}, \]
which when substituted into \eqref{eq:dderiv-ident} implies
\begin{align*}
  \frac{1}{g_z}\frac{\partial^2x_k}{\partial \overline{x}_i\partial \overline{x}_j} &= E^{ir}E^{jk}g_{r,z} + E^{ik}E^{jr}g_{r,z}\\
  &\quad\quad+g_zE^{jr}E^{ia}E^{bk}[g_{ar,b} - g_{ar,z}g_{,b}g_z^{-1}].
\end{align*}
Combined with \eqref{eq:dpgij} this implies \eqref{eq:changevar}. We note if our initial coordinates are the $\overline{x}$ coordinates then
\begin{align*}
   E_{ij}(x,g_h(x),Dg_h(x)) = g_{ij}(x,y_0,z_h) - \frac{g_{i,z}g_j}{g_z}(x,y_0,z_h) = g_z\frac{\partial \overline{x}_j}{\partial \overline{x}_i} = \delta_{ij}.
\end{align*}
This implies \eqref{eq:van-term}.

\subsection{Quantitative convexity via duality}
\label{sec:quant-conv-via}

Here we prove the first inequality in \eqref{eq:scd} assuming the second inequality. We follow the duality argument in \cite{MR2748621} simplified by the transformation in \cite{Rankin2021}. Indeed by \cite[Lemma 3]{Rankin2021} it suffices to prove the result at the origin for the generating function
\begin{align}
  \label{eq:gen-exp}
  &g(x,y,z) = x \cdot y -z + a_{i,jk}x_iy_jy_k - f(x,y,z)z,\\
\nonumber  &f(x,y,z) = b^{(1)}_{ij}x_ix_j+b^{(2)}_{ij}x_iy_j+b^{(3)}_{ij}y_iy_j + f^{(2)}(x,y,z)z,
\end{align}
where $a_{ij,k},b^{k}_{ij},f^{(2)}$ are $C^1$ functions. Moreover we assume $u$ is a  strictly $g$-convex function satisfying $u \geq 0$ and $Du(0),u(0) = 0$. We need to prove $u(x) \geq C|x|^{1+\gamma}$. Throughout the proof we use the notation $O(x^p)$ to denote any function $h(x,y,z)$ satisfying an estimate $|h(x,y,z)| \leq C|x|^p$ on a neighbourhood of the origin for $C$ depending only on $\Vert g \Vert_{C^4}$. Similarly for the notation $O(y^p)$. 

The $g^*$-transform of $u$ is $v:Yu(\Omega) \rightarrow \mathbf{R}$ defined by
\[ v(y) = g^*(x,y,u(x)),\]
for $x$ satisfying $y = Yu(x)$ and $g^*$ the dual generating function (see \cite{Rankin2021}). 
The function $v$ is $g^*$-convex with  $g^*$-support $g^*(0,\cdot,0)$ at $0$ and the $C^{1,\alpha}$ result, which holds by duality, implies $|v(y)| \leq \overline{C}|y|^{1+\alpha}$. Thus, by duality,
\begin{align}
  \label{eq:3}
   u(x) = \sup_{y}g(x,y,v(y)) \geq \sup_{y}g(x,y,\overline{C}|y|^{1+\alpha}).
\end{align}
We also note by the $C^{1,\alpha}$ estimate for $u$,  $|Y(x,u(x),Du(x))| \leq C|x|^\alpha$, so we can assume throughout that the neighbourhood over which the supremum is taken, and subsequently $|y|$, is sufficiently small.  
The supremum is obtained for $y$ satisfying
\begin{equation}
  \label{eq:max-point}
   g_y(x,y,\overline{C}|y|^{1+\alpha}) + \overline{C}(1+\alpha)g_z(x,y,\overline{C}|y|^{1+\alpha})|y|^{\alpha-1}y = 0.
\end{equation}
Now using \eqref{eq:gen-exp}
\begin{align}
\nonumber  u(x) &\geq   g(x,y,\overline{C}|y|^{1+\alpha}) \\
\nonumber  &= x \cdot y -\overline{C}|y|^{1+\alpha}+O(y^2)  -\overline{C}f(x,y,z)|y|^{\alpha+1}.\\
\label{eq:fin-ret}  & \geq x \cdot y -\overline{C}|y|^{1+\alpha}+O(y^2) +[O(x)+O(y)]|y|^{\alpha+1}.
\end{align}

Now we use \eqref{eq:max-point} to estimate $x \cdot y$ from below. First note
\begin{equation}
  \label{eq:gy-est}
   g_y(x,y,\overline{C}|y|^{1+\alpha})  = x + O(x)O(y) + [O(x)+O(y)]|y|^{1+\alpha}. 
\end{equation}
Thus substituting into \eqref{eq:max-point} and taking an inner product with $y$ we obtain (near the origin in $x,y$)
\begin{align*}
  x \cdot y = \overline{C}(1+\alpha)|g_z(x,y,\overline{C}|y|^{1+\alpha})||y|^{\alpha+1}+O(y^2) +[O(x)+O(y)]|y|^{\alpha+1},
\end{align*}
and subsequently returing to \eqref{eq:fin-ret} implies
\begin{align*}
  u(x) &\geq\overline{C}|y|^{\alpha+1} \left[(1+\alpha)|g_z(x,y,\overline{C}|y|^{1+\alpha})| -1\right]\\
  &\quad+O(y^2) +[O(x)+O(y)]|y|^{\alpha+1}.   
\end{align*}
Since $|g_z|$ is as close to $1$ as desired on a sufficiently small neighbourhood of the origin we have
\[ u(x) \geq C(\alpha)|y|^{\alpha+1}.\]
However also from \eqref{eq:max-point} and \eqref{eq:gy-est}
\begin{align*}
  |x| &= C(1+\alpha)|g_z||y|^\alpha+ O(x)O(y) + O(y^{\alpha+1})\\
  &\leq C|y|^{\alpha}.
\end{align*}
This completes the proof.

\bibliographystyle{plain}
\bibliography{../bibliography} 
\end{document}